\documentclass[12pt,titlepage]{article}
\usepackage{layout}
\usepackage{url}
\usepackage{proof}
\usepackage{mathrsfs}
\usepackage{amsmath,amsthm,amssymb}
\usepackage{cleveref}
\usepackage[dvipdfmx]{graphicx}
\usepackage[all]{xy}

\theoremstyle{definition}
\newtheorem{thm}{Theorem}
\newtheorem{prop}[thm]{Proposition}
\newtheorem{lemma}[thm]{Lemma}
\newtheorem{cor}[thm]{Corollary}
\newtheorem{defn}[thm]{Definition}

\newtheorem{rmk}[thm]{Remark}

\newcommand{\QQ}{\mathbb{Q}}
\newcommand{\NN}{\mathbb{N}}
\newcommand{\ZZ}{\mathbb{Z}}
\newcommand{\RR}{\mathbb{R}}

\newcommand{\diadic}{\mathbb{Z}[\frac{1}{2}]}

\newcommand{\form}{\widetilde{\mathbb{G}}}
\newcommand{\games}{{\mathbb{G}}}
\newcommand{\trer}{\mathrel{\mathsurround=0pt \mbox{\raisebox{1.2pt}{\tiny \textbf{\textbar }}$\rhd$}}}

\newcommand{\trel}{\mathrel{\mathsurround=0pt \mbox{$\lhd$\raisebox{1.2pt}{\tiny \textbf{\textbar }}}}}

\title{\Large{Foundations of Temperature Theory}}
\author{Katsuyuki Bando\footnote{Katsuyuki Bando, the University of Tokyo, Graduate School of Mathematical Sciences, Japan, \texttt{kbando@ms.u-tokyo.ac.jp}}, Eitetsu Ken\footnote{Eitetsu Ken, the University of Tokyo, Graduate School of Mathematical Sciences, Japan, \texttt{eitetsu@ms.u-tokyo.ac.jp}} \& Kota Morikawa\footnote{Kota Morikawa, the University of Tokyo, Graduate School of Mathematical Sciences, Japan, \texttt{ut6489@g.ecc.u-tokyo.ac.jp}}}

\begin{document}
\maketitle
 \abstract{Temperature of combinatorial games have been long studied since when Conway established the modern combinatorial game theory, and there are several variations of the concepts. In this article, we focus on one of the classical version of temperature, and give precise proofs to the fundamental claims on temperature, namely, the existence, order-preservation, and homomorphism. Besides, a general relationship between the value and the thermograph of a game is investigated.}
 \section{Introduction}
 \quad Combinatorial game theory pursues two-player games with complete information: to evaluate every one of positions of a given game, and to find which player wins the game are ones of its ultimate goals. In order to tackle these problems, there has been developed a variety of techniques and theories, and temperature theory is one of the most important among them.\\ 
  \quad Temperature of games has been studied since the beginning of the combinatorial game theory. Several kinds of the concept can be found in the classical books such as \cite{conway} and \cite{berlekamp}. The theory is utilized to calculate the mean value of a given game, and it is also a basis of the theory of thermal dissociation, analysis of concrete games (such as AMAZONS), and many other topics of the area (for details, the reader is encouraged to see \cite{siegel}).\\
  \quad In this article, we focus on one of the standard versions of temperature (defined in Definition \ref{defoftemp} below), and give thorough and yet easy-to-understand proofs of its foundations\footnote{Especially, we will apply the Intermediate Value Theorem very carefully, and make it clear why $t(G)$ should be grater than $-1$ for a non-integer $G$ (See Theorem \ref{existence}).}. Specifically, we show that temperature does exist for every game, that cooling is order-preserving, and that cooling is homomorphism. As a byproduct, a general relationship between the value and the thermograph of a game (See Lemma \ref{comparison} and Theorem \ref{integer}) is obtained. 
  
 \section{Acknowledgement}
 The authors are really grateful to the Graduate School of Mathematical Sciences, the University of Tokyo, for expanding our possibilities and providing rooms for our seminar.  
 
 \section{Notation and convention}
 We assume the results from chapter I to chapter II \S 3 of \cite{siegel} (whose arguments do not rely on temperature theory). Also, we follow the notation of \cite{siegel} unless we state otherwise. In particular, when we say ``two games are equal," it usually means that ``the game values of them are equal."   
 
 \section{Existence of Temperature}
 Recall that there is a group isomorphism between $\diadic$ and the subgroup consisted of numbers in $\games$ (the group of game values of normal-play games). With this, we regard $\diadic < \games$.
 \begin{defn}\label{defoftemp}
  For each $G \in \form$ (where $\form$ is the set of the all formal normal-play games) and $t \in \diadic$, define $\lambda_t(G), \rho_t(G) \in \diadic$, $G_t \in \form$, $t(G) \in \diadic \cup \{- \infty \}$ by induction on $\tilde{b}(G)$ (i.e. the formal birthday of $G$) as follows (the well-definedness will be verified in Theorem \ref{existence}):
  \begin{enumerate}
   \item If $G$ is equal to an integer $n$, 
   \begin{align*}
    \lambda_t(G)&:= \rho_t(G):= n \in \diadic \subset \games,\\
    G_t&:= n \in \form
   \end{align*}
 ($n$ is in canonical form, say\footnote{It can be easily seen that the exact form does not affect to the theory. It is enough that the value is determined modulo equality.}). Set $t(G):=-\infty$.
   \item Otherwise, let
    \begin{align*}
     \widetilde{\lambda}_t(G)&:= \max_{G^L} \{ \rho_t(G^L)-t\} \in \diadic\\
     \widetilde{\rho}_t(G)&:= \min_{G^R} \{ \lambda_t(G^R)+t\} \in \diadic\\
     \widetilde{G}_t&:= \{ (G^L)_t-t \mid (G^R)_t+t\} \in \form
    \end{align*}
    Then, there exists
    \begin{align*}
    t(G) := \min \left\{t \in \diadic \mid \widetilde{\lambda}_t(G)=\widetilde{\rho}_t(G) \right\}.
    \end{align*}
    Using this, set
    \begin{align*}
     \lambda_t(G)&:= \begin{cases}
      \widetilde{\lambda}_t(G) &( t \leq t(G))\\
      x &(otherwise)
     \end{cases}\\
      \rho_t(G)&:= \begin{cases}
      \widetilde{\rho}_t(G) &( t \leq t(G))\\
      x &(otherwise)
     \end{cases}\\
     G_t&:= \begin{cases}
      \widetilde{G}_t &(t \leq t(G))\\
      x_{can} &(otherwise)
     \end{cases}
    \end{align*}
    where $x:=  \lambda_{t(G)}(G) = \rho_{t(G)}(G)$, and $x_{can} \in \form$ is the canonical form of $x$.
  \end{enumerate}
 \end{defn}
   The first thing we shall show is that the above ``definition" actually works. Towards this end, we will define the following notion:
   
 \begin{defn}
  A diadic trajectory $f \colon \diadic \rightarrow \diadic$ is a function having the following form:
  \begin{align*}
   f(x)= \begin{cases}
    a_0x+b_0 &(x \leq j_0)\\
    a_ix+b_i &(j_{i-1} \leq x \leq j_i \quad \mbox{for each $i=1, \ldots n$})\\
    a_{n+1}x+b_{n+1} &(j_{n} \leq x)
   \end{cases}
  \end{align*}
  where $0 \leq n$, $j_0, \ldots, j_n$, each $b_i$ is in $\diadic$ and $a_i$ is in $\{-1,0,1\}$. Also, 
  \begin{align*}
  a_ij_i+b_i=a_{i+1}j_{i+1}+b_{i+1}.
  \end{align*}
  \quad In words, a diadic trajectory is a function whose graph is a connected union of finite segments with slope $\pm 1$ or $0$, ending with two half lines, and whose joint points are all in $\diadic$. \\
  \quad In addition, when a diadic trajectory $f$ given above satisfies $a_{n+1}=0$, the last half line $f(x)=a_{n+1}x+b_{n+1}=b_{n+1}$ is called its \it{mast},
   \rm and $b_{n+1}$ is called its 
   \it{mast value}. 
 \end{defn}
 
 \begin{rmk}\label{intermediate}
  Notice that if $f,g$ are diadic trajectories, then
  $M(x):= \max \{ f(x),g(x)\}$ and $m(x):= \min \{ f(x), g(x)\}$ are also diadic trajectories.\\
  In addition, if $f$ and $g$ meet at a point $(x_0,y_0) \in \QQ ^2$, then $x_0 \in \diadic$ (in particular, if $f(a_0) < 0 < f(a_1)$ ($a_0, a_1 \in \diadic$), then there exists $x \in \diadic$ such that $f(x)=0$. This is sometimes called Intermediate Value Theorem).
 \end{rmk}
 
 Now, our first claim can be formulated:
 
 \begin{thm}\label{existence}
 For each $G \in \form$, the following hold:
  \begin{enumerate}
   \item\label{well-definedness} $\lambda_t(G), \rho_t(G),G_t$ and $t(G)$ are well-defined, and $t(G)>-1$ if $G$ is not equal to an integer.
   \item\label{decreasing} $\lambda_t(G)$ is (as a function of $t$) a diadic trajectory consisting of segments with slope $0$ or $-1$.
   \item\label{increasing} $\rho_t(G)$ is (as a function of $t$) a diadic trajectory consisting of segments with slope $0$ or $1$.
   \item\label{mast value} Both $\lambda_t(G), \rho_t(G)$ have masts, and the mast values are equal.   
   \item\label{stops} $\forall t \in \diadic.\ L(G_t)=\lambda_t(G)$ and $R(G_t)=\rho_t(G)$. 
   \item\label{classify} Suppose $G$ is not equal to an integer. Then, for each $t>-1$,
       \begin{enumerate}
        \item If $t<t(G)$, $G_t$ is not numberish.
        \item If $t=t(G)$, $G_t$ is numberish.
        \item If $t>t(G)$, $G_t$ is not equal to a number.
       \end{enumerate}
  \end{enumerate}
 \end{thm}
 
 To prove this theorem, the following notion and lemma are crucial:
 
 \begin{defn}
  For $G \in \form$, define $\overline{\lambda}(G)$ and $\overline{\rho}(G)$ by induction on $\widetilde{b}(G)$ as follows:
  \begin{enumerate}
   \item If $G$ is equal to an integer $n \in \ZZ$ (including the case $\widetilde{b}(G)=0$), set 
   \begin{align*}
   \overline{\lambda}(G):=\overline{\rho}(G):=n
   \end{align*}
   
   \item Otherwise, set
   \begin{align*}
    \overline{\lambda}(G):= \max_{G^L} \{ \overline{\rho}(G^L)+1\}\\
    \overline{\rho}(G):= \min_{G^R} \{ \overline{\lambda}(G^R)-1\}
   \end{align*}
  \end{enumerate}
  Notice that $\overline{\lambda}(G),\overline{\rho}(G) \in \ZZ$.
 \end{defn}
 
 \begin{lemma}\label{integerstop}
  Let $G \in \form$. Then, the following holds:
  \begin{enumerate}
   \item Suppose $n \in \ZZ$ and $G$ is not equal to an integer, and $\overline{\lambda}(G) \leq n$. Then, each $G^L$ satisfies $G^L \trel n$.
   \item Suppose $n \in \ZZ$ and $G$ is not equal to an integer, and $n \leq \overline{\rho}(G)$. Then, each $G^R$ satisfies $n \trel G^R$.
   \item Suppose $n \in \ZZ$ and $G$ is not equal to an integer, and $n<\overline{\lambda}(G) $. Then, there exists a $G^L$ such that $n \leq G^L$.
   \item Suppose $n \in \ZZ$ and $G$ is not equal to an integer, and $\overline{\rho}(G) < n$. Then, there exists a $G^R$ such that $G^R \leq n$. 
  \end{enumerate} 
 \end{lemma}
 
 \begin{proof}
  The four claims are proved simultaneously by induction on $\widetilde{b}(G)$.\\
  \quad When $G$ is equal to an integer, the claims are trivial.\\
  \quad Suppose otherwise. We shall deal with the case $\overline{\lambda}(G) \leq n$ first. In this case, $\overline{\rho}(G^L)+1 \leq n$ holds for each $G^L$. Fix a $G^L$. When $G^L$ is equal to an integer, this means $G^L \trel n$. Otherwise, by IH (induction hypothesis), there exists a $G^{LR}$ such that $G^{LR}\leq n$. This means $G^{L} \trel n$. Since $G^L$ was arbitrary, we get the claim.\\
  \quad The case $n \leq \overline{\rho}(G)$ is similar.\\
  \quad Consider the case $n<\overline{\lambda}(G)$. Then, by the definition of $\overline{\lambda}(G)$, there exists a $G^L$ such that $\overline{\rho}(G^L)+1>n$. When $G^L$ is equal to an integer, this means $n \leq G^L$. Otherwise, by IH, each $G^{LR}$ satisfies $n \trel G^{LR}$, that is, $n \leq G^L$ (by the Integer Avoidance Theorem).\\
  \quad The case $\overline{\rho}(G)<n$ can be dealt with similarly.
 \end{proof}

 Now, we are ready to prove Theorem \ref{existence}.
 \begin{proof}[Proof of Theorem \ref{existence}]
 The all six claims are proved simultaneously by induction on $\widetilde{b}(G)$.\\
 \quad When $G$ is equal to an integer (including the case of $\widetilde{b}(G)=0$), the all six claims are obviously true.\\
 \quad From now on, we consider the case $G$ is not equal to any integer.\\
 \quad First, observe that $\widetilde{\lambda}_t(G)=\max_{G^L}\{\rho_t(G^L)\}-t$ is, by IH, a diadic trajectory whose segments have slope $0$ or $-1$, and for sufficient large $t$, its slope is $-1$. Hence, $\widetilde{\lambda}_t(G)$ is (as a function of $t$) non-increasing, and eventually strictly-insreasing. Similarly, $\widetilde{\rho}_t(G)=\min_{G^R}\{\lambda_t(G^R)\}+t$ is a diadic trajectory whose segments have slope $0$ or $+1$, and for sufficient large $t$, its slope is $+1$. Hence, $\widetilde{\rho}_t(G)$ is non-decreasing, and eventually strictly-increasing.\\
 \quad Now, let us show that there is a minimum $t \in \diadic$ such that $\widetilde{\lambda}_t(G)=\widetilde{\rho}_t(G)$. By the observation above and Remark \ref{intermediate}, it suffices to show that $\widetilde{\lambda}_{-1}(G)>\widetilde{\rho}_{-1}(G)$.\\
 \quad Suppose otherwise. We have $\widetilde{\lambda}_{-1}(G) \leq \widetilde{\rho}_{-1}(G)$. 
 By IH, each proper subposition $H$ of $G$ is not frozen at $t=-1$ (i.e. $t(H)>-1$) unless $H$ is equal to an integer. Hence, using induction from the leaves to the root of the game tree of $G$, we can see that $\widetilde{\lambda}_{-1}(G)=\overline{\lambda}(G) \in \ZZ$ and $\widetilde{\rho}_{-1}(G)= \overline{\rho}(G) \in \ZZ$. Let $x:=\widetilde{\lambda}_{-1}(G)$, $y:=\widetilde{\rho}_{-1}(G)$. It is enough to show $G^L \trel x \leq y \trel G^R$ for every $G^L$ and $G^R$ (Then, by the Simplicity Theorem, it follows that $G$ is equal to an integer, which is impossible. As for the Simplicity Theorem, see \cite{siegel}), and this is an immediate consequence of Lemma \ref{integerstop}.\\
 \quad Now, the six claims can be verified as follows;\\
 \quad Consider (\ref{well-definedness}). Combining the above result and the monotonicity of $\widetilde{\lambda}_t(G)$ and $\widetilde{\rho}_t(G)$, it follows that there exists a minimum $-1< t \in \diadic$ such that $\widetilde{\lambda}_t(G)=\widetilde{\rho}_t(G)$, which means $t(G)>-1$ is well-defined. Therefore, $\lambda_t(G), \rho_t(G), G_t$ are all well-defined.\\
 \quad (\ref{decreasing})(\ref{increasing})(\ref{mast value}) are immediate consequences of the first observation above and the definitions of $\lambda_t(G)$ and $\rho_t(G)$.\\
 \quad (\ref{stops}) follows from a simple calculation:
 when $t>t(G)$, the claim is obvious. Otherwise,
 \begin{align*}
  \lambda_t(G)= \widetilde{\lambda}_t(G)= \max_{G^L} \{ \rho_t(G^L) -t\}
  = \max_{G^L} \{ R((G^L)_t-t) \}=L(\widetilde{G}_t)=L(G_t)
 \end{align*} 
 
 Here, the third equality is by IH and a basic property of stops (namely, $L(G+x)=L(G)+x$ if $x$ is equal to a number. See \cite{siegel} for reference). $\rho_t(G)=R(G_t)$ can be proved similarly.\\
 \quad (\ref{classify}) is now proved as follows. If $t<t(G)$, $\lambda_t(G) =\widetilde{\lambda}_t(G) > \widetilde{\rho}_t(G) =\rho_t(G)$ by the definition of $t(G)$, which means (by (\ref{stops})) $L(G_t)>R(G_t)$. This shows that $G_t$ is not numberish. If $t \geq t(G)$, $L(G_t)=\lambda_t(G)=\rho_t(G)=R(G_t)$, and hence $G_t$ is numberish. Especially when $t>t(G)$, $G_t$ is a number by the definition.  
 
 \end{proof}
 
 \section{More Precise Behavior of Temperature}
 The proof above does not tell us whether (When $t(G)> -1$) the numberish $G_{t(G)}$ is actually a number or not. In fact, it is not equal to a number. Also, it is interesting to ask what happens to $\widetilde{G}_t$ when $t>t(G)$, and in fact, it equals to a number. This section deals with these issues.\\
  \quad We begin with the following crucial lemma, which is powerful when combined with the Simplicity theorem:
 
 \begin{lemma}\label{comparison}
  Let $G \in \form$ and suppose that $G$ is not equal to any integer.
  \begin{enumerate}
   \item If $\widetilde{\lambda}_t(G)$ is a constant for $t \in (a,b]$ (where $a<b \in \diadic$), then, for each $t \in (a,b]$, $G^L_t-t \trel \widetilde{\lambda}_t(G)$ for any $G^L$ \footnote{Here, $(G^L)_t$ is denoted by $G^L_t$. The other notations (such as $G^R_t$, $G^{LR}_t$) follow the similar convention.}. 
   \item If $\widetilde{\rho}_t(G)$ is a constant for $t \in (a,b]$ (where $a<b \in \diadic$), then, for each $t \in (a,b]$, $\widetilde{\rho}_t(G) \trel G^R_t+t$ for any $G^R$.
   \item If $\widetilde{\lambda}_t(G)$ is strictly decreasing for $t \in (a,b]$ (where $a<b \in \diadic$), then, for each $t \in (a,b]$, there exists a $G^L$ such that $G^L_t-t \geq \widetilde{\lambda}_t(G)$.  
   \item If $\widetilde{\rho}_t(G)$ is strictly increasing for $t \in (a,b]$ (where $a<b \in \diadic$), then for each $t \in (a,b]$, there exists a $G^R$ such that $\widetilde{\rho}_t(G) \geq G^R_t+t$.  
  \end{enumerate}
 \end{lemma}
 
 \begin{proof}
  The four claims are proved simultaneously by induction on $\widetilde{b}(G)$.\\
  \quad When $G$ is equal to an integer, the claims are obvious.\\
  \quad Suppose otherwise. Let us consider the case $\widetilde{\lambda}_t(G)$ is a constant $x$ for $t \in (a,b]$ (where $a<b \in \diadic$) first.
  We have to show that $G^L_t -t \trel x$ for each $G^l$ and $t \in (a,b]$. Fix a $G^L$, and $t \in (a,b]$. If $\rho_t(G^L) -t = R(G^L_t -t) < x$, then $G^L_t -t \trel x$ follows from a basic property of stops. Consider the case $\rho_t(G^L) -t = R(G^L_t -t) = x$. In this case, the slope of $\rho_s(G^L)$ on $s \in (t-\delta, t] \subset (a,b]$ ($\delta$ is a sufficiently small positive number such that the slope is uniquely determined) must be $+1$ (Otherwise, $\rho_s(G^L) -s > x$ holds for $s \in (t-\delta, t]$. Therefore, combined with the definition of $\widetilde{\lambda}_s(G)$, we get $\widetilde{\lambda}_s(G) >x$ for $s \in (t-\delta, t]$, which contradicts to the definition of $x$). And therefore, it follows that 
  \begin{align*}
  t \leq t(G^L),\ \  \rho_t(G^L)=\widetilde{\rho}_t(G^L), \ \ G^L_t=\widetilde{G^L}_t .
  \end{align*}
  Hence, IH can be applied, and we obtain a $G^{LR}$ which satisfies $\widetilde{\rho}_t(G^L)\geq G^{LR}_t+t$.
  It means 
  \begin{align*}
  x=\rho_t{G^L}-t=\widetilde{\rho}_t(G^L)-t \geq (G^{LR}+t)-t,
  \end{align*}
  and this gives $G^L_t-t \trel x$.\\
  \quad The case $\widetilde{\rho}_t(G)$ is a constant for $t \in (a,b]$ (where $a<b \in \diadic$) can be dealt with similarly.\\
  \quad Consider the case $\widetilde{\lambda}_t(G)$ is strictly decreasing for $t \in (a,b]$ (where $a<b \in \diadic$). Fix $t \in (a,b]$. The definition of $\widetilde{\lambda}_s(G)$ and the monotonicity of $\rho_s(G^L)$'s imply that there exists a $G^L$ which attains, for $s \in (t-\delta,t]$, $\rho_s(G^L)-s=\widetilde{\lambda}_s(G)$, and therefore $\rho_s(G^L)$ is constant (where $\delta$ is a sufficiently small positive number).
  Then, for such a $G^L$, $G^L_t-t \geq \widetilde{\lambda}_t(G)$ can be proved as follows:
  when $G_t$ is not equal to a number, then $t \leq t(G^L)$, and the argument above shows $\widetilde{\rho}_t(G^L)=\rho_t(G^L)$ is constant. Hence, IH implies (note that $G^L$ is not equal to an integer since $-\infty<t \leq t(G^L)$) $\widetilde{\rho}_t(G^L)\trel G^{LR}_t+t$ for each $G^{LR}$. It means $\widetilde{\lambda}_t(G)=\widetilde{\rho}_t(G^L)-t \trel (G^{LR}_t+t)-t$ for each $G^{LR}$. This gives $G^L_t-t \geq \widetilde{\lambda}_t(G)$ (by the Integer Avoidance Theorem). Now, consider the case in which $G^L_t$ is equal to a number. Then, recalling how we took the $G^L$, we get 
  \begin{align*}
  \widetilde{\lambda}_t(G)=\rho_t(G^L)-t= R(G^L_t)-t=G^L_t-t.
  \end{align*}
  This completes the proof of $G^L_t-t \geq \widetilde{\lambda}_t(G)$.\\
  \quad The remaining case can be treated similarly. 
 \end{proof}
 
 Now, we can prove the main claims of this section.
 
 \begin{prop}\label{tepid}
  Assume $G \in \form$, and $G$ is not equal to any integer. Then, $G_{t(G)}$ is not equal to a number (although it is numberish by Theorem \ref{existence}).
 \end{prop}
 
 \begin{proof}
  Suppose otherwise. Then, $G_{t(G)}=L(G_{t(G)})=\widetilde{\lambda}_t(G)$, and $G_{t(G)}=R(G_{t(G)})=\widetilde{\rho}_t(G)$. Now, since $t=t(G)$ is the minimum $t$ such that $\widetilde{\lambda}_t(G)=\widetilde{\rho}_t(G)$, at least one of the diadic trajectories $\widetilde{\lambda}_t(G)$ or $\widetilde{\rho}_t(G)$ has non-zero slope on $t\in (t(G)-\delta,t(G)]$ (where $\delta$ is a sufficiently small positive number). Consider the case $\widetilde{\lambda}_t(G)$ has non-zero slope. Then, it is strictly decreasing on $t \in (t(G)-\delta,t(G)]$, so by Lemma \ref{comparison}, there exists $G^L$ such that $G^L_{t(G)}-t(G) \geq \widetilde{\lambda}_t(G_{t(G)})=G_{t(G)}$. Since the LHS is a left option of $G_{t(G)}$, together with the Simplicity theorem, it contradicts the assumption $G_{t(G)}$ is equal to a number.\\
  \quad The case $\widetilde{\rho}_t(G)$ has non-zero slope can be treated symmetrically, and hence, the claim is proved. 
 \end{proof}
 
 \begin{prop}\label{Gtildenumber}
 Let $G \in \form$, and $t>t(G)$. Then, 
  $\widetilde{G}_t$ is equal to a number.
 \end{prop}
 \begin{rmk}
  Be aware that $\widetilde{G}_t$ above can be different from $G_t$.
 \end{rmk}
 
 \begin{proof}
  When $G$ is equal to an integer, the claim is obvious. 
  Suppose otherwise. Let 
  \begin{align*}
  x:=\widetilde{\lambda}_{t(G)}(G)=\widetilde{\rho}_{t(G)}(G).
  \end{align*}
   By the Simplicity Theorem, it suffices to show $G^L_t-t \trel x$ and $x \trel G^R_t+t$ for any $G^L, G^R$ and $t >t(G)$. Fix a $G^L$ and $t>t(G)$. Since $\widetilde{\lambda}_{t(G)}(G)=x$, the monotonicity of $\widetilde{\lambda}_{s}(G)$ gives $\widetilde{\lambda}_{t}(G) \leq x$, which gives 
   \begin{align*}
   R(G^L_t-t) = \rho_t(G^L)-t \leq x.
   \end{align*}
    If $R(G^L_t-t)<x$, then $G^L_t-t \trel x$ follows from a basic property of stops. If $R(G^L_t-t)=x$, then the monotonicity of $\widetilde{\lambda}_{t}(G)$ and the fact $\widetilde{\lambda}_{t(G)}(G)=x$ imply that $\widetilde{\lambda}_{s}(G)$ is constant for $s \in [t(G),t]$. Hence, by the Lemma \ref{comparison}, $G^L_t-t \trel x$. $G^L$ was arbitrary.\\
  \quad The claim $x \trel G^R_t+t$ can be proved symmetrically.  
 \end{proof}

 \section{$(\cdot)_t$ Defines a Homomorphism from $\games$ to $\games$}
  Now, we will prove the two main theorems, that is, the operation $(\cdot)_t$ is order-preserving (hence, invariant with respect to $=$, and induces a map from $\games$ to $\games$), and the induced map is a group homomorphism.
 
  \begin{thm} \label{orderpreserving}
   For each $-1< t \in \diadic$, $G,H \in \form$, the following hold:
   \begin{enumerate}
    \item $G \geq H \Longrightarrow G_t \geq H_t$.
    \item $G \trer H \Longrightarrow G_t \trer H_t-t$.
   \end{enumerate}
  \end{thm}
\begin{proof}
The two claims are proved simultaneously by induction on $\widetilde{b}(G)+\widetilde{b}(H)$.\\
\quad If both $G$ and $H$ are equal to integers, the claims are easy.
Suppose otherwise.\\
\quad Assume that $G\geq H$ and consider the game $G_t-H_t$.
We know $G^R\trer H$ for any $G^R$ and $G\trer H^L$ for any $H^L$, so IH implies that $G^R_t+t-H_t\trer 0$ and $G_t-H^L_t+t\trer 0$.\\
\quad If $-1<t\leq \min\{t(G),t(H)\}$, then $G_t-H_t= \widetilde{G}_t-\widetilde{H}_t$ and its Right options are $G^R_t+t-H_t(\trer 0)$ or $G_t-H^L_t+t(\trer 0)$, so Left can win.\\
\quad If $t(G)<t\leq t(H)$, then $G_t$ is a number, and $H_t=\widetilde{H}_t$ is not equal to a number by Proposition \ref{tepid}.
By the Number Avoidance Theorem, we do not have to cosider Right's move on $G_t$.
If Right moves on $H_t$, Left can win since $G_t-H^L_t+t \trer 0$.
The same argument works in the case $t(H)<t\leq t(G)$.\\
\quad In the remaining case $\max\{t(G),t(H)\}<t$, let $t_0:=\max\{t(G),t(H)\}$.
We know $t_0>-1$ because either $G$ or $H$ is not equal to an integer, and we also know $G_{t_0}\geq H_{t_0}$ as above.
Since $G_{t_0},H_{t_0}$ are infinitesimally close to the numbers $G_t,H_t$ respectively, we obtain the desired inequality $G_t\geq H_t$.\\
\quad Next, we will prove the second claim, assume that $G\trer H$ and consider the game $G_t-H_t+t$.
We know $G^L\geq H$ for some $G^L$ or $G\geq H^R$ for some $H^R$, so IH implies that $G^L_t-H_t\geq 0$ for some $G^L$ or $G_t-H^R_t\geq 0$ for some $H^R$.\\
\quad If $-1<t\leq \min\{t(G),t(H)\}$, then $G_t-H_t+t= \widetilde{G}_t-\widetilde{H}_t+t$ and its Left options are $G^L_t-t-H_t+t=G^L_t-H_t$ or $G_t-H^R_t-t+t=G_t-H^R_t$, so left can win.\\
\quad Let $t(G)<t\leq t(H)$.
In this case, $G_t$ is a number, and $H_t=\widetilde{H}_t$ is not equal to a number by Proposition \ref{tepid}.
If $G$ is originally equal to an integer, then $G\trer H$ implies $G\geq H^R$ for some $H^R$ by the Integer Avoidance theorem.
Therefore $G_t-H^R_t\geq 0$ for some $H^R$.
In the game $G_t-H_t+t=G_t-\widetilde{H}_t+t$, Left can move to 
\begin{align*}
G_t-H^R_t-t+t=G_t-H^R_t\geq 0,
\end{align*}
 so left can win.
If $G$ is not equal to an integer, then $t(G)>-1$, so we know $G_{t(G)}\trer H_{t(G)}-t(G)$ by the above case. 
If $G_t>\rho_{t}(H)-t$, then we get $G_t\trer H_t-t$ by basic properties of stops.
If $G_t=\rho_{t}(H)-t$, then $\rho_s(H)-s$ is a constant for $s\in (t(G),t]$, so $\rho_s(H)$ is strictly increasing for $s\in (t(G),t]$.
By Lemma \ref{comparison}, $G_t=\rho_t(H)\geq H^R_t+t$ for some $H^R$.
This implies $G_t\trer H_t-t$.
The same argument works in the case $t(H)<t\leq t(G)$.\\
\quad  In the remaining case $\max\{t(G),t(H)\}<t$, let $t_0:=\max\{t(G),t(H)\}$.
We know $t_0>-1$ because either $G$ or $H$ is not equal to an integer, and we also know $G_{t_0}\trer H_{t_0}-t_0$ as above.
Since $G_{t_0},H_{t_0}$ is infinitesimally close to the number $G_t,H_t$ respectively, we obtain $G_t\geq H_t-t_0$ and the desired inequality $G_t\trer H_t-t$.
  \end{proof}
  
  
  \begin{thm} \label{homomorphism}
   For each $-1< t \in \diadic$, $G,H \in \form$, 
   \begin{align*}
    (G+H)_t=G_t+H_t
   \end{align*}
   holds.
  \end{thm}
\begin{proof}
Induction on $\widetilde{b}(G)+\widetilde{b}(H)$.
If both $G$ and $H$ are equal to integers, the claim is obvious.
Suppose otherwise.
Without loss of generality, we can assume that $t(G)\leq t(H)$, and hence $H$ is not equal to an integer.
First, we will show $\widetilde{(G+H)}_t=G_t+H_t$ for $t\in (-1, t(H)]$.
If $-1<t\leq t(G)$, then $G_t= \widetilde{G}_t$ and $H_t=\widetilde{H}_t$, so
\begin{align*}
\widetilde{(G+H)}_t&=\{(G^L+H)_t-t,(G+H^L)_t-t\mid (G^R+H)_t+t,(G+H^R)_t+t\}\\
&=\{G^L_t+H_t-t,G_t+H^L_t-t\mid G^R+H_t+t,G+H^R_t+t\}\ (\text{by induction})\\
&=G_t+H_t.
\end{align*}
If $t(G)<t\leq t(H)$, then $G_t$ is a number and $H_t=\widetilde{H}_t$ is not equal to a number, so by the Number Translation Theorem, 
\[
G_t+H_t=\{G_t+H^L_t-t\mid G_t+H^R_t+t\}
\]
Since $G^L\trel G$ always holds, we have $G^L_t-t\trel G_t$ by Theorem \ref{orderpreserving}, and hence $G^L_t+H_t-t\trel G_t+H_t$.
By Gift Horse Principle, we have
\[
G_t+H_t=\{G^L_t+H_t-t, G_t+H^L_t-t\mid G_t+H^R_t+t\}.
\]
Similarly, we can prove
\[
G_t+H_t=\{G^L_t+H_t-t, G_t+H^L_t-t\mid G^R_t+H_t+t,G_t+H^R_t-t\}.
\]
It means $\widetilde{(G+H)}_t=G_t+H_t$.\\
\quad We have obtained $\widetilde{(G+H)}_t=G_t+H_t$ for $t\in (-1, t(H)]$.
Hence, it is enough to show $\widetilde{(G+H)}_t=(G+H)_t$.
If $t\leq t(G+H)$, it is clear.
If $t(G+H)<t\leq t(H)$, then $\widetilde{(G+H)}_t=G_t+H_t$ is equal to a number by Proposition \ref{Gtildenumber}, and $H_t$ is not equal to a number by Proposition \ref{tepid}, so $G_t$ is not equal to a number.
It means $t\leq t(G)$.
Let 
\[
\mathcal{I}:=\left\{ x\in \diadic \mid (\widetilde{(G+H)}_t)^L\trel x\trel (\widetilde{(G+H)}_t)^R \text{for every} (\widetilde{(G+H)}_t)^L,(\widetilde{(G+H)}_t)^R \right\}.
\]
By the proof of Proposition \ref{Gtildenumber}, we know $(G+H)_t\in \mathcal{I}$.
By the Simplicity Theorem, it suffices to show $\#\mathcal{I}\leq 1$.
Suppose $x,y\in \mathcal{I}$ and $x<y$.
$G_t+H^L_t-t\trel x$ implies
\begin{align*}
R(G_{t(G)})+R(H^L_{t(G)}-t(G))&\leq R(G_{t(G)}+H^L_{t(G)}-t(G))\\
&\leq R((G+H^L)_{t(G)}-t(G))\ (\text{by induction})\\
&\leq R((G+H^L)_t-t)\\
&\leq R(G_t+H^L_t-t)\ (\text{by induction})\\
&\leq x.
\end{align*}
Similarly, $y\trel G_t+H^R_t-t$ implies $y\leq L(G_{t(G)})+L(H^R_{t(G)}+t(G))$.
Because $x<y$ and $R(G_{t(G)})=L(G_{t(G)})$, we obtain $R(H^L_{t(G)}-t(G))<L(H^R_{t(G)}+t(G))$.
Since $H^L$ and $H^R$ are arbitrary, it means $\widetilde{\lambda}_{t(G)}(H)<\widetilde{\rho}_{t(G)}(H)$, that is, $t(H)<t(G)$.
Contradiction.\\
\quad We have proved that $(G+H)_t=G_t+H_t$ for $t\in (-1, t(H)]$ so far.
The remaining case is $t>t(H)$.
Let $t_0:=t(H)$.
We know $t_0>-1$ because $H$ is not equal to an integer, $(G+H)_{t_0}=G_{t_0}+H_{t_0}$ as above.
Since $G_{t_0},H_{t_0}$ are numberish, $(G+H)_{t_0}$ is also numberish.
Therefore, $G_{t_0},H_{t_0},(G+H)_{t_0}$ are infinitesimally close to $G_t,H_t,(G+H)_t$ respectively.
It implies $(G+H)_t=G_t+H_t$.
\end{proof}  
  
  \section{Consequences}
  The claims and proofs other than Theorem\ref{integer} in this section essentially follow those in \cite{siegel}, but we will present them in our words for completeness.
  
  \begin{prop}\label{G_0=G}
   $G_0 = G$.
  \end{prop}
  \begin{proof}
  Induction on $\widetilde{b}(G)$. If $G$ is equal to an integer, the claim is easy. Suppose otherwise. Then,
  \begin{align*}
  \widetilde{G}_0&=\{G^L_0-0 \mid G^R_0+0\}\\
  &=\{G^L \mid G^R\}\ (\text{by induction})\\
  &=G.
  \end{align*}
  \quad Hence, if $t(G)\geq 0$, then $G_0=\widetilde{G}_0=G$.\\
  \quad If $t(G)< 0$, by Proposition \ref{Gtildenumber}, $\widetilde{G}_0$ is equal to a number, and therefore, so is $G$. Since $G$ is not equal to an integer, it can be written as 
  \begin{align*}
  G= \frac{k}{2^m}
  \end{align*} 
(where $k$ is odd, and $m\geq 1$). Let $H \in \form$ be the canonical form of $\frac{k}{2^m}$. Recall that $H$ can be written as follows:
 \begin{align*}
 H:=\left\{\frac{k-1}{2^m} \mid \frac{k+1}{2^m} \right\}
 \end{align*}
 (where $\frac{k-1}{2^m}$ and $\frac{k+1}{2^m}$ are again in their canonical forms). By Theorem \ref{orderpreserving} and $G=H$, it suffices to show that $H_0=\frac{k}{2^m}$.\\
  \quad We prove above by showing the following: For each $m \geq 1$ and an odd natural number $k$, the canonical form $H$ of $\frac{k}{2^m}$ satisfies $t(H)=-\frac{1}{2^m}$ and $H_{t(H)}$ is $\frac{k}{2^m}$-ish. \\
  \quad Induction on $m \geq 1$. Fix $m \geq 1$. We can assume that $(\frac{k\pm 1}{2^m})_t = \frac{k\pm 1}{2^m}$ with $t>-\frac{1}{2^{m-1}}$ (for $m =1$, this is obvious, and for $m>1$, this follows from IH). Then,
  \begin{align*}
  \widetilde{H}_{-\frac{1}{2^m}}&=\left\{\left(\frac{k-1}{2^m}\right)_{-\frac{1}{2^m}} +\frac{1}{2^m} \mid \left(\frac{k+1}{2^m}\right)_{-\frac{1}{2^m}} -\frac{1}{2^m}\right\}\\
  &=\left\{\frac{k-1}{2^m} +\frac{1}{2^m} \mid \frac{k+1}{2^m} -\frac{1}{2^m}\right\}\\
  &=\left\{\frac{k}{2^m} \mid \frac{k}{2^m}\right\}\\
  &=\frac{k}{2^m}+\{0 \mid 0\}\ (\text{by the Number Translation Theorem})\\
  &=\frac{k}{2^m}+\ast
  \end{align*}
  This is $\frac{k}{2^m}$-ish but not equal to a number. Therefore, by Theorem \ref{existence} (\ref{classify}) and Proposition \ref{Gtildenumber}, $t(H)=-\frac{1}{2^m}$, and $H_{t(H)}=\widetilde{H}_{t(H)}$ is $\frac{k}{2^m}$-ish.
  \end{proof}
 
  \begin{cor}\label{tempclass}
  For each $G \in \form$, the following holds:
   \begin{enumerate}
    \item $t(G)<0 \Longleftrightarrow G$ is equal to a number.
    \item $t(G)=0 \Longleftrightarrow G$ is numberish but not equal to any number.
    \item $t(G)>0 \Longleftrightarrow G$ is not numberish.
   \end{enumerate}
  \end{cor}
  \begin{proof}
  Each right arrow is shown by Theorem \ref{existence}, Proposition \ref{tepid}, and \ref{G_0=G}. Hence, each left arrow can also be proved because the sign of $t(G)$ and the state of $G$ can take just one of three patterns stated above.
  \end{proof}
  
  \begin{rmk}
  In particular, if $G\in\diadic$ and $t\geq 0$, then $G_t=G_0=G$.
  \end{rmk}
  \begin{cor}
   $t(G+H) \leq \max \{t(G),t(H)\}$.
  \end{cor}
  \begin{proof}
  If both $G_t$ and $H_t$ are numberish, $(G+H)_t$ is also numberish by Theorem \ref{homomorphism}. Using this, the claim follows immediately.
  \end{proof}
  
  \begin{cor}
   Let $G,H \in \form$. Also, assume that $t, u \in \diadic$, $t>-1$, and $u \geq 0$. Then, 
   \begin{align*}
    (G_t)_u=G_{t+u}.
   \end{align*}
  \end{cor}
  \begin{proof}
  Induction on $\widetilde{b}(G)$. If $G$ is equal to an integer, the claim is obvious. If $t > t(G)$, then $G_t$ and $G_{t+u}$ are equal to the same number, so the claim is shown easily. Now, suppose that $G$ is not equal to an integer, and $t \leq t(G)$. \\
  \quad If $u \leq t(G_t)$,
   \begin{align*}
    (G_t)_u&=\{ (G_t)^L_u - u \mid (G_t)^R_u + u \}\\
    &= \{ (G^L_t - t)_u - u \mid (G^R_t + t)_u + u \}\\
    &= \{ (G^L_t)_u - t_u - u \mid (G^R_t)_u + t_u + u \}\ (\text{by Theorem \ref{homomorphism}})\\
    &= \{ G^L_{t+u}-t-u \mid G^R_{t+u}+t+u \}\ (\text{by induction})\\
    &= \widetilde{G}_{t+u}
   \end{align*}
   Especially when $u = t(G_t)$, $(G_t)_u$ is numberish but not equal to any number, and so is $\widetilde{G}_{t+u}$. Therefore, we can get $t(G)=t+t(G_t)$, and $(G_t)_u=\widetilde{G}_{t+u}=G_{t+u}$ when $u \leq t(G_t)$.\\
   \quad In the remaining case $u > t(G_t)$, since $t(G)=t+t(G_t)$, both $(G_t)_u$ and $G_{t+u}$ are infinitesimally close to $(G_t)_{t(G_t)} = G_{t+t(G_t)}$. Because $(G_t)_u$ and $G_{t+u}$ are equal to numbers, the two games are equal to the same number.
  \end{proof}
  
  The next corollary is very famous; it gives us how to calculate the mean value of a given game.
  
  \begin{cor}
   Suppose $t>t(G)$. Then, $G_t=m(G)$ i.e. 
   \begin{align*}
   \lim_{n\rightarrow \infty}\frac{L(nG)}{n}=G_t=\lim_{n\rightarrow \infty}\frac{R(nG)}{n} \quad \mbox{in $\RR$}
   \end{align*} 
   (where $nG:= \sum_{i=1}^n G$).
   More precisely, There exists a constant (which can be dependent on $t,G$) $C \in \diadic$ such that for each $n \in \NN$, $|nG-nG_t| \leq C$ in $\games$ holds. 
  \end{cor}
  \begin{rmk}
   The game $G_t$ above is sometimes denoted by $G_\infty$. Since it is a number, the statement above tells that there exists an asymptotic approximation of $G$ by a number $G_\infty$ (notice that it is easy to see that such an approximation should be unique).
  \end{rmk}
  \begin{proof}
  If $G$ is equal to a number, the claim is trivial. \\
  \quad Suppose otherwise. In this case, $t>t(G)\geq 0$. By Theorem \ref{homomorphism}, $(nG)_t=nG_t$, so 
  \begin{align*}
  L((nG)_t)=nG_t=R((nG)_t).
  \end{align*}
   Considering the slopes of $\widetilde{\lambda}_t(nG)$ and $\widetilde{\rho}_t(nG)$ (which are determined by Theorem \ref{existence}) and Proposition \ref{G_0=G}, 
  \begin{align*}
  L((nG)_t)&\geq L((nG)_0)-t = L(nG)-t > L(nG)-t-1, \\
  R((nG)_t)&\leq R((nG)_0)+t = R(nG)+t < R(nG)+t+1
  \end{align*}
   (note that we have $t>0$ by Corollary \ref{tempclass}). Therefore, we get $L(nG) < nG_t +t+1$ and $R(nG) > nG_t-t-1$. Because $nG_t\pm(t+1)$ is equal to a number, by the property of stops, $nG_t-t-1<nG<nG_t+t+1$. Therefore, the claim is shown by choosing $C=t+1$.
  \end{proof}
  
  Lastly, we will show the following inductive characterization of integers using thermographs (it is easy to extract a feasible algorithm for finding the thermograph of a given formal game $G \in \form$, simultaneously deciding whether each subposition of $G$ is equal to an integer or not).
  
  \begin{thm}\label{integer}
   Suppose $G \in \form$, and consider the following conditions for $G^L$'s and $G^R$'s (set $\max \emptyset:=-\infty$ and $\min \emptyset:=+\infty$):
   \begin{enumerate}
    \item\label{int1} $\max_{G^L}\{\rho_0(G^L)\} \in \ZZ$, and $G^L$ is equal to it.
    \item\label{slope1} $\max_{G^L}\{\rho_0(G^L)\} \in \ZZ$, and satisfies $\rho_0(G^L)=\max_{G^L}\{\rho_0(G^L)\}$ although $G^L$ is not equal to an integer. Furthermore, $\widetilde{\rho}_s(G^L)$ has slope $0$ on $s \in (-\delta,0]$, where $\delta$ is a sufficiently small positive number such that the slope is uniquely determined.
    \item\label{int2} $\min_{G^R}\{\lambda_0(G^r)\} \in \ZZ$, and $G^R$ is equal to it.
    \item\label{slope2} $\min_{G^R}\{\lambda_0(G^r)\} \in \ZZ$, and $G^R$ satisfies $\lambda_0(G^R)=\min_{G^R}\{\lambda_0(G^R)\}$ although $G^R$ is not equal to an integer. Furthermore, $\widetilde{\rho}_s(G^R)$ has slope $0$ on $s\in(-\delta,0]$, where $\delta$ is a sufficiently small positive number such that the slope is uniquely determined.
   \end{enumerate}
   
   Let
   \begin{align*}
    l&:= \begin{cases}
    \max_{G^L}\{\rho_0(G^L)\}+1 \quad &(\mbox{if there exists a $G^L$ which satisfies the condition (\ref{int1}) or (\ref{slope1})})\\
    \max_{G^L}\{\rho_0(G^L)\} \quad &(\mbox{otherwise})\\
    \end{cases}\\
    r&:= \begin{cases}
     \min_{G^R}\{\lambda_0(G^R)\}-1 \quad &(\mbox{if there exists a $G^R$ which satisfies the  condition (\ref{int2}) or (\ref{slope2})})\\
    \min_{G^R}\{\lambda_0(G^R)\} \quad &(\mbox{otherwise})\\
    \end{cases}
   \end{align*} 
   and 
   \begin{align*}
    S:=\{ n \in \ZZ \mid l \leq n \leq r\}
   \end{align*}
   Then, whether $G$ is equal to an integer can be decided as follows:
   \begin{enumerate}
    \item If $S=\emptyset$, then $G$ is not equal to an integer.
    \item Otherwise, $G$ is equal to $n \in S$ whose absolute value attains the minimum of 
    \begin{align*}
    \{|m| \mid m \in S \}
    \end{align*}
     (notice that such $n$ is unique since $S$ is an interval of $\ZZ$).
   \end{enumerate}
  \end{thm} 
  
  \begin{proof}
   By the Simplicity Theorem, it suffices to show that $S$ is equal to the set $I$ defined as follows:
   \begin{align*}
    I:=\{n \in \ZZ \mid \forall G^L.\ G^L \trel n \ \& \ \forall G^R.\ n \trel G^R\}
   \end{align*}
   Let us show $S \subset I$ first. Fix $n \in S$. By symmetry, we will only show $\forall G^L.  G^L\trel n$. If there is no $G^L$, the claim is trivial, so suppose otherwise. Fix a $G^L$. By the definition of $l$, $R(G^L)=\rho_0(G^L) \leq n$ holds. If $R(G^L)$ is in particular strictly less than $n$, then the claim follows from the basic property of stops. So, we will focus on the case $R(G^L)=n$. Because $R(G^L) \leq l \leq n$, $R(G^L)=l$. Together with the definition of $l$, it follows that $G^L$ does not satisfy the condition (\ref{int1})(\ref{slope1}). Therefore, $G^L$ is not equal to any integer, and $\widetilde{\rho}_s(G^L)$ is strictly increasing for $s \in (-\delta, 0]$. Hence, Lemma \ref{comparison} can be applied, and it follows that there exists a $G^{LR}$ such that $n=R(G^L)=\widetilde{\rho}_0(G^L) \geq G^{LR}_0+0=G^{LR}$. Thus, we get $n \trer G^L$. This completes the proof of $S \subset I$.\\
   \quad Now, we will show $I \subset S$. Fix $n \in I$. By symmetry, we will focus on showing $l \leq n$. \\
   \quad If there exists a $G^L$ which satisfies condition (\ref{int1}), then 
   \begin{align*}
   l=\max_{G^L}\{\rho_0(G^L)\}+1=G^L+1.
   \end{align*}
    Besides, $n \in I$ gives $G^L \trel n$, so $G^L+1 \leq n$ (since the both sides are equal to integers), which implies $l \leq n$.\\ 
   \quad If there exists a $G^L$ which satisfies condition (\ref{slope1}), then 
   \begin{align*}
   l=\max_{G^L}\{\rho_0(G^L)\}+1=\rho_0(G^L)+1.
   \end{align*} 
   Also, $n \in I$ gives $G^L \trel n$, which implies $\rho_0(G^L)=R(G^L) \leq n$. Hence, in order to show $l\leq n$, it suffices to show $R(G^L)<n$ (note that $R(G^L)$ is equal to an integer). Suppose otherwise. Then, because the condition (\ref{slope1}), Lemma \ref{comparison} can be applied, and it follows that 
   \begin{align*}
   R(G^L)=\widetilde{\rho}_0(G^L) \trel G^{LR}_0+0=G^{LR}
   \end{align*}
    for any $G^{LR}$ (the first equality can be obtained as follows: $R(G^L)=n \in \ZZ$ and the fact that $G^L$ is not equal to an integer imply that $G^L$ is not equal to a number. Therefore, $0 \leq t(G)$ and $R(G^L)=\rho_0(G^L)=\widetilde{\rho}_0(G^L)$). This means $n=R(G^L) \leq G^L$ (by the Integer Avoidance Theorem), which is a contradiction.\\
   \quad If there is no $G^L$ satisfying neither of condition (\ref{int1})(\ref{slope1}), then $l=\max_{G^L}\{\rho_0(G^L)\}$. Because $n \in I$ gives $\rho_0(G^L) \leq n$ as above, we obtain $l \leq n$. This completes the proof of $I \subset S$.         
  \end{proof}

\end{document}